\documentclass[10pt]{amsart}
\usepackage{latexsym,amsmath,amssymb}
 \renewcommand{\epsilon}{\varepsilon}
 \newcommand{\newsection}[1]
  {\section{#1}\setcounter{theorem}{0} \setcounter{equation}{0}\par\noindent}
   \newtheorem{theorem}{Theorem}[section]
   \newtheorem{lemma}[theorem]{Lemma}
 \newtheorem{corr}[theorem]{Corollary}
 
 \newtheorem{proposition}[theorem]{Proposition}
 \newtheorem{deff}[theorem]{Definition}
 \newtheorem{remark}[theorem]{Remark}

  \numberwithin{equation}{section}

 \newcommand{\bth}{\begin{theorem}}
 \newcommand{\ble}{\begin{lemma}}
 \newcommand{\bcor}{\begin{corr}}
 \newcommand{\bdeff}{\begin{deff}}
 \newcommand{\bprop}{\begin{proposition}}
 \def\be{\begin{equation}}
\def\ee{\end{equation}}
\def\bt{\begin{theorem}}
\def\et{\end{theorem}}
\def\ba{\begin{array}}
\def\ea{\end{array}}
\def\bl{\begin{lemma}}
\def\el{\end{lemma}}
\def\ddiv{\mathrm{div}}
 \newcommand{\ele}{\end{lemma}}
 \newcommand{\ecor}{\end{corr}}
 \newcommand{\edeff}{\end{deff}}
 
 \newcommand{\eprop}{\end{proposition}}

 \renewcommand{\Pi}{\varPi}

 \renewcommand{\epsilon}{\varepsilon}

\title[ A Blow-up criterion for 3-D compressible visco-elasticity]
{ A Blow-up criterion for 3-D compressible visco-elasticity}
\author{Yi Du}
\author{Chun Liu}
\author{Qingtian Zhang}
\address{Department of Mathematics, South China Normal University, duyidy@gmail.com}
\address{Department of Mathematics,Penn. State University, liu@math.psu.edu}
\address{Department of Mathematics,Penn. State University, qingtian.zh@gmail.com}

\date{}

\begin{document}
\maketitle

\begin{abstract}
In this paper, we prove a blow-up criterion for 3D compressible visco-elasticity
in terms of the upper bound of the density and the deformation tensor.
\end{abstract}

\newsection{Introduction}

Visco-elastic fluids exhibit the characteristics for both fluid and solid. The elastic behavior of these materials
is attributed to the underlying microstructure or configurations,
the deformation of the structures will cause the  exchange of kinetic energy and
elastic internal energy. The energy exchange is realized through a coupling of
the transport of the internal elastic variables and the induced elastic stress.

To illustrate this coupled interaction, we denote $x(t,X)$ as the Eulerian coordinate, where $X$ is the Lagrangian coordinate of the
particles. The velocity field defined on the Eulerian coordiante  be
\begin{equation}\label{1.1}
u(t,x)=x_t(t,X(t,x)),
\end{equation}
and the deformation tensor is
\begin{equation}\label{1.2}
F(t,x)=H(t,X(t,x))=\frac{\partial x}{\partial X}(t, X),
\end{equation}
which is determined by a
transport equation (see \cite{Gurtin,Larson,Lin1}),
 \begin{equation}\label{1.3}
F_t+u\cdot \nabla F=\nabla uF.
\end{equation}
This equation is regarded as the compatibility condition between velocity $u$ and deformation tensor $F$.
For homogenous, hyper-elastic and isotropic material, the action function is (see \cite{Lei4,Lin1,Liu1})
 \begin{equation}\label{1.4}
A(x)=\int_0^T\int_{\Omega_0}\frac12\rho_0(X)|x_t(t,X)|^2-E(H)-P(\rho(x(X,t),t))\det H dXdt,
\end{equation}
where $\rho_0(X)$ is the density in the undeformed configuration, $\rho$ is the density in the deformed configuration. $\Omega_0$ is the original domain occupied by the material. $E(H)$ is the elastic energy function. $P(\rho)$ is the hydrostatic pressure, satisfying $P(\rho)=C_0\rho^{\gamma}$, with $C_0>0$ and the adiabatic
index $\gamma>1$. $P(\rho)\det H$ represents the internal energy.
Taking a variation of $A(x)$ with respect to the flow map $x(t,X)$,  we get the inviscid momentum equation in Eulerian coordinates (\cite{Lei2})
 \begin{equation}\label{1.5}
\rho(u_t+u\cdot\nabla u)+\nabla P=\nabla\cdot \big(\frac{1}{\det F}\frac{\partial E(F)}{\partial F}F^{T}\big),
\end{equation}
where $F^T$ is the transpose of F, the
 term $\frac{1}{\det F}\frac{\partial E(F)}{\partial F}F^{T}$
is the Cauchy-Green tensor.
When we take the dissipation energy into consideration,
see \cite{Ericksen,Doi}, we can get the corresponding viscosity momentum equation
  \begin{equation}\label{1.6}
\rho(u_t+u\cdot\nabla u)+\nabla P=\mu \Delta u+(\mu+\lambda)\nabla \ddiv u
+\nabla\cdot \big(\frac{1}{\det F}\frac{\partial E(F)}{\partial F}F^{T}\big).
\end{equation}
Combining \eqref{1.3}, \eqref{1.6} and mass conservation equation, with Hookean elasticity,  we get the following  viscoelasticity system
\begin{equation}\label{equation1}
\begin{cases}
\rho_t+\nabla\cdot(\rho u)=0,
\\
\rho(u_t+u\cdot\nabla u)+\nabla P=\mu \Delta u+(\mu+\lambda)\nabla \ddiv u+\nabla\cdot \big(\frac{1}{\det F}FF^{T}\big),
\\
F_t+u\nabla F=\nabla uF.
\end{cases}
\end{equation}
The viscosity coefficient $\mu$, $\lambda$ are constants satisfying
$
\mu\geq0,  \  3\lambda+2\mu\geq0.
$

The system \eqref{equation1}, when $F$ do not exist, it is a compressible Navier-Stokes equations.
The global strong solutions existence for large initial data is still open. Therefore,
there are many efforts to study the blow-up criteria. See \cite{Huang1,zhangzhifei1} and references therein.

When we take the deformation tensor $F$ into consideration, the system \eqref{equation1} will be much more complicate, although the deformation
tensor admits a transport equation which is similar to the continuity equation.
The incompressible visco-elasticity has been  studied by \cite{Lei6,Lei2,Lei4,Lei1,Lei3,Lin1,Liu1} etc.
Precisely, in 2D, the local strong solutions to large data and global strong solutions to small data
 have been studied in \cite{Lei2,Lei4,Lei1,Lei3,Lin1}. The global strong solutions in 3D  with small data have been presented in \cite{Lei1}.
 Subsequently, the blow-up criteria for 2D Oldroyd model have been presented in \cite{Lei5}, a blow-up criterion in 3D  with partial viscosity 
 was proved in \cite{Lei7}.

For the compressible case, The local strong solutions with large data  and global
strong solutions  with small data  to \eqref{equation1} were proved in \cite{Hu1} and \cite{Hu2,Qian1} respectively.
Subsequently, Hu-wang \cite{Hu5}, presented a blow-up criterion for local strong solutions as $\|\nabla u\|_{L^1(0,T,L^\infty(\mathbb{R}^3))}<\infty$. To proceed, we firstly introduce some notations as following
\begin{equation}\nonumber
D^{k}(\mathbb{R}^3)=\{u\in L_{loc}^1(\mathbb{R}^3):\|\nabla^k u\|_{L^2(\mathbb{R}^3) }\},
\end{equation}
\begin{equation}\nonumber
D_0^k(\mathbb{R}^3)=\{u\in L^6(\mathbb{R}^3):\|\nabla^k u\|_{L^2(\mathbb{R}^3)}<\infty
\}.
\end{equation}
In Hu-Wang\cite{Hu5}, they have proved the following results,
\begin{theorem}(Hu-Wang \cite{Hu5}).\label{theorem1}
Assume that the initial data satisfy
$0\leq \rho_0\in H^3(\mathbb{R}^3)$, $u_0\in D_0^1(\mathbb{R}^3)\bigcap D^3(\mathbb{R}^3)$, $F_0 \in H^3(\mathbb{R}^3),
\nabla\cdot(\rho_0F_0)=0$ and
\begin{equation}\nonumber
-\mu \Delta u_0-(\lambda+\mu)\nabla \ddiv u_0+A\nabla \rho_0^\gamma=\rho_0g,
\end{equation}
for some $g\in H^1(\mathbb{R}^3)$ with $\sqrt{\rho_0}g\in L^2(\mathbb{R}^3)$. There exist classical
solutions $(\rho,u,F)$  to (\ref{equation1}) satisfying
\begin{equation}\label{1.9}
\begin{cases}
(\rho,F)\in C([0,T^\star],H^3(\mathbb{R}^3)),\\
u\in C([0,T^\star],D_0^1(\mathbb{R}^3)\bigcap D^3(\mathbb{R}^3))\bigcap L^2(0,T^\star, D^4(\mathbb{R}^3)),\\
u_t\in L^\infty([0,T^\star,D_0^1(\mathbb{R}^3))\bigcap L^2(0,T^\star, D^2(\mathbb{R}^3)),\\
\sqrt{\rho}u_t\in L^\infty([0,T^\star,L^2(\mathbb{R}^3)),
\end{cases}
\end{equation}
where $T^\star$ is the maximal existence time. If $T^\star< \infty$ and $7\mu>\lambda$, then,
\begin{equation}
\lim_{T\rightarrow T^\star}\int_0^T\|\nabla u\|_{L^\infty(\mathbb{R}^3)}dt=\infty.
\end{equation}
\end{theorem}

Enlightened by the work of Sun-Wang-Zhang \cite{zhangzhifei1}, in which the authors presented the blow-up criterion for a
compressible Navier-Stokes in the terms  of $\|\rho\|_{L^\infty(0,T,L^\infty)}$, we will establish a blow-up criterion
in terms of the upper bounds of the density and
deformation tensor for the local strong solution to
the 3D compressible visco-elasticity.
 Our main result states as follows.
\begin{theorem}\label{theorem2}
Assume that $(\rho,u,F)$ is the local strong solution mentioned in \eqref{1.9}, and
$\mu,\lambda$ be as in Theorem \ref{theorem1}. The initial data $\rho_0>\epsilon_0>0$, $u_0$ and $F_0$  as in Theorem \ref{theorem1},
 $T^\star$ is the maximal existence time of the
solution. If $T^\star<\infty$, then we have
\begin{equation}\label{1.10}
\lim_{T\uparrow T^\star}\sup\{\|\rho(t)\|_{L^\infty(0,T;L^\infty(\Omega))}
+\|F\|_{L^\infty(0,T;L^\infty(\Omega))}\}=\infty.
\end{equation}
\end{theorem}
\begin{remark}
According to \eqref{2.1}, the condition \eqref{1.10} implies that, if $\|\rho\|_{L^\infty(0,T,L^\infty(\Omega))}=\infty$,
or equivalently $det F=0$, the mass accumulate at one point or the material  volume compress to zero.
If  $\|F\|_{L^\infty(0,T,L^{\infty}(\Omega))}=\infty$,then the deformation leads to
 the  blow-up phenomenon.

\end{remark}

If  $\|\rho(t)\|_{L^\infty(0,T;L^\infty(\Omega))}
$ and  $\|F\|_{L^\infty(0,T;L^{\infty}(\Omega))}
$ are bounded,
we can obtain a high integrability of velocity, which can be used to control the
nonlinear term (See Lemma 3.2). The difficulty is to control the density and deformation, which
satisfy transport equations. To do this, it requires the velocity is bounded in
$L^1(0,T;W^{1,\infty}(\Omega))$. On the other hand, we have to obtain some
priori bounds for  $\nabla \rho$ and $\nabla F$ to prove $u\in L^1(0,T;W^{1,\infty}(\Omega))$, furthermore,
the elasticity term $\nabla\cdot(\frac{1}{\det F}FF^T)$ in the momentum equation  will bring extra
difficulty to us. Thanks to
the structure of the equations, we get the cancelation to the derivatives of $\rho, F$ during our computation,
which brings us the desired result. Moreover, we get the result without the restriction
on $\rho^{-1}$, it is  an advantage compared with Sun-Wang-Zhang \cite{zhangzhifei3}.

In this paper, we follow the line of Sun-Wang-Zhang \cite{zhangzhifei1,zhangzhifei3}, write
\begin{equation}
L\triangleq\mu \Delta  +(\lambda +\mu)\nabla \ddiv
\end{equation}
then, the Lam\'{e} operator $L$ is an elliptical operate.

Generally speaking, the fluid is driven by three kinds of forces which present different mechanisms. They are gradient of  pressure, elastic deformation, and inertial force. According to these mechanisms, we can decompose the velocity field to several parts. Here, since the pressure term and the elasticity term have similar form in the equation, we put them together and decompose velocity field into two parts, that is
\be\label{1.12}
u=v+w,
\ee
with
\be\label{1.13}
Lv=\nabla p-\nabla\cdot(\frac{1}{\det F} FF^T),
\ee
and
\be\label{1.14}
Lw=\rho \dot{u},
\ee
where
$$ \dot{u}=\partial_t u+u\cdot \nabla u. $$
By the decomposition, in order to obtain the regularity of $u$, it is sufficient to consider the regularity of $v$ and $w$, which admit the above elliptic equations.

Our paper is organized as following. In Section 2,
we shall present some preliminaries. Section 3 is to present
some priori estimates. The proof of Theorem 1.2 should be presented
in Section 4.

%%%%%%%%%%%%%%%%%%%%%%%%%%%%%%%%%%%%%%%%%%%%%%%%%%%%%%%%%%%%%%%%%%%%%%%%%%%%%%%%%%%%%%%%%%%%%%%%%%%%%%%%%%%%%%%%%%%%%%%%%%%%%%%%%%%%%%%%%

\section{Preliminaries}

Consider the following boundary value problem for the Lam\'{e} operator $L$
\be\label{lam2.1}
\left\{
\ba{l}
\mu\triangle U+(\mu+\lambda)\nabla\ddiv U=F, \quad \text{ in }\Omega\\
U(x)=0, \quad \text{ on }  \partial \Omega
\ea
\right.\ee
where $\Omega$ is $\mathbb{R}^3$ or a bounded domain in $\mathbb{R}^3$.
We present the following Lemmas, see  \cite{zhangzhifei1,zhangzhifei3} for details.
\bl\label{lem2.1}
Let $q\in(1,\infty)$ and $U$ be a solution of (\ref{lam2.1}). There exists a constant $C$ depending only on $\lambda$, $\mu$, $q$ and $\Omega$ such that the following estimates hold.

(1) If $F\in L^q(\Omega)$, then
\be
\left\{
\ba{l}
\|D^2U\|_{L^q(\mathbb{R}^3)}\leq C\|F\|_{L^q(\mathbb{R}^3)},\\
\|U\|_{W^{2,q}(\Omega)}\leq C\|F\|_{L^q(\Omega)}; \quad \text{if $\Omega$ is a bounded domain.}
\ea\right.
\ee

(2) If $F\in W^{-1,q}(\Omega)$(i.e., $F=\ddiv f$ with $f=(f_{ij})_{3\times3}$, $f_{ij}\in L^q(\Omega)$), then
\be\left\{
\ba{l}
\|DU\|_{L^q(\mathbb{R}^3)}\leq C\|f\|_{L^q(\mathbb{R}^3)},\\
\|U\|_{W^{1,q}(\Omega)}\leq C\|f\|_{L^q(\Omega)}; \quad \text{if $\Omega$ is a bounded domain.}
\ea\right.
\ee

(3) If $F=\ddiv f$ with $f_{ij}=\partial_k H_{ij}^k$ and $h_{ij}^k\in W_0^{1,q}(\Omega)$ for $i,j,k=1,2,3$, then
\be
\|U\|_{L^q(\Omega)}\leq C\|h\|_{L^q(\Omega)}.
\ee
\el

\bl\label{lam2.2}
If $F=\ddiv f$ with $f=(f_{ij})_{3\times3}$, $f_{ij}\in L^\infty(\Omega)\bigcap L^2(\Omega)$, then $\nabla U\in BMO(\Omega)$ and there exists a constant $C$ depending only on $\lambda$, $\mu$ and $\Omega$ such that
\be
\|\nabla U\|_{BMO(\Omega)}\leq C(\|f\|_{L^\infty(\Omega)}+\|f\|_{L^2(\Omega)}).
\ee
\el

\bl\label{lam2.3}
Let $\Omega=\mathbb{R}^3$ or be a bounded Lipschitz domain and $f\in W^{1,q}(\Omega)$ with $q\in(3,\infty)$. There exists a constant $C$ depending on $q$ and the Lipschitz property of $\Omega$ such that
\be
\|f\|_{L^\infty(\Omega)}\leq C(1+\|f\|_{BMO(\Omega)})\ln(e+\|\nabla f\|_{L^q(\Omega)}).
\ee
\el

The following lemma is well-known.
\begin{lemma}\label{lam2.4}
 Let $\rho, F$ as defined before, and $\rho_0$ is
the initial data of $\rho$, then
\begin{equation}\label{2.1}
\rho\cdot \det F=\rho_0.
\end{equation}
\end{lemma}

In the following, we use the notation, for the matrix $(A)_{3\times3}$ and $(B)_{3\times3}$
\be\nonumber
A:B=A_{ij}B_{ij},
\ee
where summation applied to terms with repeated index.

\section{Priori Estimate}
By the standard energy estimates, we have
\bl\label{lem3.1}(Priori estimate)
\[
\|\rho(t)\|_{L^1(\Omega)}= \|\rho_0\|_{L^1(\Omega)},
\]
\begin{multline}
\|\sqrt{\frac{\rho}{\rho_0}} F\|_{L^2(\Omega)}+\|\rho(t)\|_{L^\gamma(\Omega)}^\gamma+\|\rho|u|^2(t)\|_{L^1(\Omega)}+\|\nabla u\|_{L^2((0,t)\times\Omega)}^2
\\
\leq C(\|\rho_0\|_{L^\gamma(\Omega)}^\gamma+\|\rho_0|u_0|^2\|_{L^1(\Omega)}+\|F_0\|_{L^2(\Omega)}).
\end{multline}
\el

\begin{lemma}\label{lem3.2}
Assume the initial data as in Theorem \ref{theorem2},  $7\mu>\lambda$,  the density $\rho$ and deformation tensor $F$ satisfy
\be\label{3.1}
\|\rho\|_{L^\infty(0,T;L^\infty(\Omega))}+\|F\|_{L^\infty(0,T;L^{\infty}(\Omega))}< \infty,\footnote{Recall Lemma 2.4 and $\rho_0>\epsilon_0>0$, this condition implies $0<c_0<detF<\infty, a.e$.}
\ee
then, there exists $r\in (3,6)$, such that $\rho|u|^r\in L^\infty(0,T;L^1(\Omega))$, with
\be
\|\rho|u|^r\|_{L^\infty(0,T;L^1(\Omega))}\leq C.
\ee
Here $C$ depends on $T$, $\|\rho\|_{L^\infty(\Omega)}$, $\|F\|_{L^\infty(\Omega)}$ and initial data.
\end{lemma}
\begin{proof}
Multiplying the equation $(\ref{equation1})_2$ by $r|u|^{r-2}u$ and integrating the resulting equation on $\Omega$ to obtain
\begin{multline}\label{3.4}
\frac{d}{dt}\int_\Omega\rho|u|^rdx+\int_\Omega r|u|^{r-2}(\mu|\nabla u|^2+(\lambda+\mu)(\ddiv u)^2)\\
+r(r-2)(\mu|u|^{r-2}|\nabla|u||^2+(\lambda+\mu)\ddiv u|u|^{r-3}u\cdot\nabla|u|)dx\\
=\int_\Omega r P(\rho)\ddiv(|u|^{r-2}u)dx-\int_\Omega r\nabla(|u|^{r-2}u):(\frac{1}{\det F} FF^T)dx.
\end{multline}
By using the fact $|\nabla u|\geq |\nabla|u||$, the term in second integrand can be estimated by
\begin{align}\label{3.5}
r|u|^{r-2}[\mu|\nabla u|^2 &    +(\lambda+\mu)(\ddiv u)^2  +(r-2)\mu|\nabla|u||^2-(\lambda+\mu)(r-2)|\nabla|u||\ddiv u]\\\nonumber
\geq & r|u|^{r-2}[\mu|\nabla u|^2+(r-2)(\mu-(\lambda+\mu)\frac{r-2}{4})|\nabla|u||^2]
\\\nonumber
\geq & C|u|^{r-2}|\nabla u|^2.
\end{align}
The pressure term
\be\label{3.6}
\int_{\Omega}P(\rho)\ddiv (|u|^{r-2}u)dx\leq C\int_\Omega\rho^{\frac{r-2}{2r}}|u|^{r-2}|\nabla u|dx,
\ee
and the elasticity term
\be\label{3.7}
\int_\Omega|\nabla(|u|^{r-2}u):\frac{1}{\det F} FF^T|dx\leq C\int_\Omega(|u|^{r-2}|\nabla u|\rho^{\frac{r-2}{2r}})dx.
\ee
By using
\be\label{3.8}
\int_\Omega \rho^{\frac{r-2}{2r}}|u|^{r-2}|\nabla u|dx\leq \varepsilon\int_\Omega|u|^{r-2}|\nabla u|^2dx+\frac{C}{\varepsilon}(\int_\Omega \rho|u|^rdx)^{\frac{r-2}{r}},
\ee
then, \eqref{3.4}-\eqref{3.7} imply the desired estimate.
\end{proof}

\begin{proposition}\label{proposition3.3}
Under the assumption \eqref{3.1},
then we  have
\begin{equation}\label{Prop3.3}
\|\nabla w\|_{L^\infty(0,T,L^2(\Omega))}, \|\rho^{\frac12}\partial_tw\|_{L^2(0,T,\Omega)},\|\nabla^2 w\|_{L^2(0,T,\Omega)} \leq C,
\end{equation}
where $C$ is constant depends on $\|\rho\|_{L^\infty(0,T,L^\infty(\Omega))}$, $\|F\|_{L^\infty(0,T,L^\infty(\Omega))}$ and
the initial data.
\end{proposition}
\begin{proof}
Using the momentum equation, we get
\begin{equation}
\begin{cases}
\rho\partial_t w-L w=\rho G,\quad in\  [0,T)\times \Omega,\\
w(t,x)=0,\quad on\  [0,T)\times \partial \Omega,\quad w(0,x)=w_0(x) \quad in\  \Omega,
\end{cases}
\end{equation}
where
\begin{align}\label{prop3.3.1}
G&=   -u\mbox{div}u-L^{-1}\nabla(\partial_tP)+L^{-1}\nabla\partial_t(\frac{1}{\det F} FF^T)
\\\nonumber
&=   -u\mbox{div}u +L^{-1}\nabla\mbox{div}(Pu)+L^{-1}\nabla[(\rho P'(\rho)-P) \mbox{div}u]
\\\nonumber
& \quad +L^{-1}\nabla[\nabla u\cdot (\frac{1}{\det F} FF^T)+\frac{1}{\det F} FF^T(\nabla u)^T-\mbox{div}(u\otimes\frac{1}{\det F} FF^T)].
\end{align}
Multiplying the equation with $\partial_t w$ and integrating  over $\Omega$,
with the H\"{o}lder inequality, we get
\begin{equation}
\frac{d}{dt}\int_\Omega \mu |\nabla w|^2+(\lambda+\mu)|\mbox{div}w|^2dx+
\frac12\int_\Omega\rho|\partial_t w|^2dx \leq \frac12 \|\sqrt{\rho}G \|_{L^2(\Omega)}^2.
\end{equation}
Now, we shall estimate   $\|\sqrt{\rho}G \|_{L^2(\Omega)}^2$ term by term.
\begin{align}
\|\sqrt{\rho}u\mbox{div}u\|_{L^2(\Omega)} &\leq C\|\rho^{\frac1r}\|_{L^r(\Omega)}\|\nabla u\|_{L^{\frac{2r}{r-2}}(\Omega)}
\\\nonumber
 &\leq C(\epsilon)\|\nabla w\|_{L^2(\Omega)}+\epsilon \|\nabla^2 w\|_{L^2(\Omega)}+C,
\end{align}
Here $2\leq r<6$, and we used the interpolation inequality
\begin{equation}
\|\cdot\|_{L^r(\Omega)}\leq C(\epsilon)\|\cdot\|_{L^2(\Omega)}+\epsilon \|\nabla \cdot\|_{L^2(\Omega)}.
\end{equation}
From the estimates for Lam\'{e} operator and the energy estimates Lemma 3.1, we have
\begin{equation}
\|\sqrt{\rho}L^{-1}\nabla\mbox{div}(Pu)\|_{L^2(\Omega)}\leq C
\|Pu\|_{L^2(\Omega)}\leq C\|\sqrt{\rho}u\|_{L^2(\Omega)}\leq C,
\end{equation}
and
\begin{align}\label{prop3.3.2}
\|\sqrt{\rho}L^{-1}\nabla[\nabla u\cdot      &        (\frac{1}{\det F} FF^T)
+\frac{1}{\det F} FF^T(\nabla u)^T-\nabla\cdot (u\otimes \frac{1}{\det F} FF^T)]\|_{L^2(\Omega)}
\\\nonumber
 \leq& C
\|\sqrt{\rho}\|_{L^3(\Omega)}\|L^{-1}\nabla[\nabla u\cdot (\frac{1}{\det F} FF^T)
+\frac{1}{\det F} FF^T(\nabla u)^T]\|_{L^6(\Omega)}
\\\nonumber
&+C\|L^{-1}\nabla[\nabla\cdot (u\otimes \frac{1}{\det F} FF^T)]\|_{L^2(\Omega)}
 \\\nonumber
 \leq & C
\|\nabla u\|_{L^2(\Omega)}
+C.
\end{align}
Similarly, we have
\begin{equation}
\|\sqrt{\rho}L^{-1}  \nabla[(\rho P'-P)\mbox{div}u]\|_{L^2(\Omega)}
\leq C
\|\nabla u\|_{L^2(\Omega)}.
\end{equation}
Then, the results \eqref{Prop3.3} follows from \eqref{prop3.3.1}-\eqref{prop3.3.2} and energy estimates Lemma 3.1.
\end{proof}

\begin{corr}
Under the assumption of Lemma \ref{lem3.2}, we have
\begin{equation}
\|\nabla u\|_{L^\infty(0,T,L^2)}, \|\nabla u\|_{L^2(0,T,L^q(\Omega))} \leq C,
\end{equation}
for any $q\in [2,6]$.
\end{corr}

\section{Proof of Theorem 1.2}
\begin{proposition}\label{prop1}
Suppose $T^\star<\infty$ is the maximal existence time, $\forall T,  0\leq  T< T^\star$ and $3<q<6$, if
\be\label{4_4}
\|\rho,\rho^{-1}\|_{L^\infty(0,T, L^\infty(\Omega))}+\|F\|_{L^\infty(0,T, L^{\infty}(\Omega))}<\infty,
\ee
then, we obtain
\be\label{4_5}
\|\nabla^2 w\|_{L^2(0,T;L^q(\Omega))}\leq C,
\ee
 and
\be\label{4_6}
\int_0^T\|\nabla u\|_{L^\infty}ds<\infty.
\ee
\end{proposition}

\begin{proof}
We first estimate $ \|\nabla^2 w\|_{L^2(0,T;L^q(\Omega))}$.
By Lemma \ref{lem2.1} and \eqref{1.14}, we have
\begin{equation}\label{4.4}
\|\nabla^2 w\|_{L^q(\Omega)}\leq \|\rho\dot{u}\|_{L^q(\Omega)}.
\end{equation}
Noting,
\be\label{u}
\rho\dot{u}- Lu+\nabla P=\nabla\cdot(\frac{1}{\det F} FF^T),
\ee
taking $\partial_t$ to this equation, we get
\be\label{4_1}
\rho_t \dot{u}+\rho\dot{u}_t-Lu_t+\nabla P_t=\nabla\cdot(\frac{1}{\det F} FF^T)_t.
\ee
Applying $u\otimes$  and divergence to (\ref{u}), we have
\be\label{4_2}
\nabla\cdot(\rho u)\dot{u}+\rho u\cdot\nabla\dot{u}-\nabla\cdot(u\otimes Lu)+\nabla\cdot(u\otimes\nabla P)=\nabla\cdot[u\otimes\nabla\cdot(\frac{1}{\det F} FF^T)].
\ee
Adding \eqref{4_1} and \eqref{4_2} together, we get
\begin{multline}
 \rho \dot{u}_t+\rho u\cdot\nabla\dot{u}+\nabla P_t+\nabla\cdot[u\otimes\nabla P]\\
 =\mu[\triangle u_t+\nabla\cdot(u\otimes\triangle u)]
+(\lambda+\mu)[\nabla\ddiv u_t+\ddiv(u\otimes\nabla\ddiv u)]\\
+\nabla\cdot[(\frac{1}{\det F} FF^T)_t+u\otimes\nabla\cdot(\frac{1}{\det F} FF^T)].
\end{multline}
Multiplying above equation by $\dot{u}$ and integrating on $\Omega$, we get
\begin{multline}\label{4.10}
\frac{d}{dt}\int_\Omega\frac{1}{2}\rho|\dot{u}|^2dx-\mu\int_\Omega(\triangle u_t+\ddiv(u\otimes\triangle u))\cdot \dot{u}dx
\\-(\lambda+\mu)\int_\Omega(\nabla\ddiv u_t+\ddiv(u\otimes\nabla\ddiv u))\cdot \dot{u}dx\\
=\int_\Omega P_t\ddiv \dot{u}dx
+\int_\Omega u\cdot\nabla\dot{u}\cdot\nabla pdx
\\+\int_\Omega \nabla\cdot[(\frac{1}{\det F} FF^T)_t+u\otimes\nabla\cdot(\frac{1}{\det F} FF^T)]\cdot\dot{u}dx.
\end{multline}
The estimates of the second and the third terms on the left hand side, as well as  the first and the second terms on the right hand side
in \eqref{4.10}, are the same as in \cite{zhangzhifei1,zhangzhifei3}. For completeness, we give a brief proof as following
\begin{align}
-\int_\Omega(\triangle u_t &+\ddiv(u\otimes\triangle u))\cdot\dot{u}dx
=\int_\Omega\nabla u_t:\nabla\dot{u}+u\otimes\triangle u:\nabla\dot{u} dx\\\nonumber
=&\int_\Omega|\nabla\dot{u}|^2-\nabla(u\cdot\nabla u):\nabla\dot{u}+u\times\triangle u:\nabla\dot{u}dx\\\nonumber
=&\int_\Omega[|\nabla\dot{u}|^2-(\nabla u\nabla u):\nabla\dot{u}+((u\cdot\nabla)\nabla\dot{u}):\nabla u
\\\nonumber
 &-(\nabla u\nabla\dot{u}):\nabla u-((u\cdot\nabla)\nabla\dot{u}):\nabla u]dx\\\nonumber
\geq & \int_\Omega[\frac{3}{4}|\nabla\dot{u}|^2-C|\nabla u|^4]dx,
\end{align}
and
\begin{align}
-\int_\Omega &(\nabla\ddiv u_t      +\ddiv(u\otimes\nabla\ddiv u))\cdot \dot{u}dx\\\nonumber
=&\int_\Omega[|\ddiv\dot{u}|^2-\ddiv\dot{u}\nabla u:(\nabla u)^T-\ddiv u(\nabla\dot{u})^T:\nabla u+\ddiv \dot{u}(\ddiv u)^2]dx\\\nonumber
\geq&\int_\Omega[\frac{1}{2}|\ddiv\dot{u}|^2-\frac{1}{4}|\nabla\dot{u}|^2-C|\nabla u|^4]dx.\\\nonumber
\end{align}
We continue to estimate the pressure term.
\begin{align}
\int_\Omega P_t\ddiv \dot{u}&+(u\cdot\nabla\dot{u})\cdot\nabla Pdx\\\nonumber
=&\int_\Omega-\rho P'(\rho)\ddiv u\ddiv \dot{u}+P[\ddiv u\ddiv\dot{u} - (\nabla u)':(\nabla\dot{u})]dx\\\nonumber
\leq& C\|\nabla u\|_{L^2(\Omega)}\|\nabla\dot{u}\|_{L^2(\Omega)}\leq C\|\nabla\dot{u}\|_{L^2(\Omega)}.
\end{align}

Next, we estimate the elasticity term $\int_\Omega \nabla\cdot[(\frac{1}{\det F} FF^T)_t+u\otimes\nabla\cdot(\frac{1}{\det F} FF^T)]\cdot\dot{u}dx$.
From (\ref{equation1}), we get
\begin{align}\label{4.15}
&\int_\Omega \nabla\cdot   [(\frac{1}{\det F} FF^T)_t +u\otimes\nabla\cdot(\frac{1}{\det F} FF^T)]\cdot\dot{u}dx\\\nonumber
=& \int_\Omega \nabla\cdot[\nabla u\cdot(\frac{1}{\det F} FF^T)+\frac{1}{\det F} FF^T(\nabla u)^T-\frac{1}{\det F} FF^T(\nabla\cdot u)]\dot{u}dx\\\nonumber
&+\int_\Omega\nabla\cdot[u\otimes\nabla\cdot(\frac{1}{\det F} FF^T)-u\cdot\nabla(\frac{1}{\det F} FF^T)]\dot{u}dx
\end{align}
Since $\|\rho\|_{L^\infty(0,T,L^\infty(\Omega))}$, $\|F\|_{L^\infty(0,T,L^\infty(\Omega))}<\infty$, we get the bound for  the  first three terms on the right hand side of the above equation as following
\begin{equation}
|\int_\Omega \nabla\cdot[\nabla u\cdot(\frac{1}{\det F} FF^T)+\frac{1}{\det F} FF^T(\nabla u)^T-\frac{1}{\det F} FF^T(\nabla\cdot u)]\dot{u}dx|
\leq C\int_\Omega |\nabla u||\nabla\dot{u}|dx.
\end{equation}
As to the last term on the right hand side of \eqref{4.15}, we have a cancelation
\begin{align}
 \nabla\cdot[u\otimes\nabla\cdot&(\frac{1}{\det F} FF^T)-u\cdot\nabla(\frac{1}{\det F} FF^T)]\cdot\dot{u}\\\nonumber
=&-\partial_i[\partial_j u_i(\frac{1}{\det F} F_{jk}F_{lk})]\dot{u}_l+\partial_i[\partial_j u_j(\frac{1}{\det F} F_{jk}F_{lk})]\dot{u}_l.
\end{align}
Therefore,
\begin{align}
 \int_\Omega \nabla\cdot[u\otimes\nabla\cdot(\frac{1}{\det F} FF^T)   &   -u\cdot\nabla(\frac{1}{\det F} FF^T) ]\cdot\dot{u}dx\\\nonumber
\leq& C\|\nabla\dot{u}\|_{L^2(\Omega)}\|\nabla u\|_{L^2(\Omega)}.\\\nonumber
\end{align}
Combining the above estimates,
recalling Lemma \ref{lem3.1}, we get
\begin{align}\label{4_3}
\frac{d}{dt}\int_\Omega\rho|\dot u|^2dx+ \int_\Omega|\nabla\dot u|^2dx\leq C(1+\|\nabla u\|^4_{L^4(\Omega)}).
\end{align}
Noting  that
\begin{align}
 \|\nabla u\|_{L^4(\Omega)}^4&\leq\|\nabla u\|_{L^2(\Omega)}\|\nabla u\|_{L^6(\Omega)}^3\\\nonumber
&\leq C\|\nabla u\|_{L^6(\Omega)}^2(\|\nabla w\|_{L^6(\Omega)}+\|\nabla v\|_{L^6(\Omega)})\\\nonumber
&\leq C\|\nabla u\|_{L^6(\Omega)}^2(1+\|\dot u\|_{L^2(\Omega)}).
\end{align}
Substituting this estimate into (\ref{4_3}) and using Corollary 3.4, we have $\|\nabla u(t)\|_{L^6(\Omega)}^2\in L^1(0,T)$,
then conclude by Gronwall's inequality that
\be
\|\sqrt{\rho}\dot u\|_{L^\infty(0,T;L^2(\Omega))}+\|\nabla\dot u\|_{L^2((0,T)\times\Omega)}\leq C.
\ee
Recalling \eqref{4.4} and Lemma \ref{lem2.1}, (4.2) is verified.

In the following, we shall prove (4.3). Noting that
\begin{align}
 \|\nabla u\|_{L^1(0,T;L^\infty(\Omega))}&\leq C\|u\|_{L^1(0,T;W^{2,q}(\Omega))}\\\nonumber
&\leq C(\|v\|_{L^1(0,T;W^{2,q}(\Omega))}+\|w\|_{L^1(0,T;W^{2,q}(\Omega))}),
\end{align}
where $3< q < 6$. By Lemma \ref{lem2.1}, $\|v\|_{W^{2,q}(\Omega)}\leq\|\nabla P\|_{L^q(\Omega)}+\|\nabla\cdot(\frac{1}{\det F} FF^T)\|_{L^q(\Omega)}$, from the assumption \eqref{4_4} and lemma 2.4, we have
\be
\|\nabla P\|_{L^q(\Omega)}\leq C\|\nabla \rho\|_{L^q(\Omega)},
\ee
and
\begin{align}
&\|\nabla\cdot(\frac{1}{\det F}FF^T)\|_{L^q(\Omega)}\\\nonumber
\leq& C\|\nabla \rho\|_{L^q(\Omega)}\|F\|_{L^\infty}^2+C\|\rho\|_{L^\infty(\Omega)}\|\nabla F\|_{L^q(\Omega)}\|F\|_{L^\infty}\\\nonumber
\leq& C\|\nabla\rho\|_{L^q(\Omega)}+C\|\nabla F\|_{L^q(\Omega)}.
\end{align}
Therefore, recalling \eqref{4_5}, to bound $\|\nabla u\|_{L^1(0,T,L^\infty(\Omega))}$,  it suffices  to bound $\|\nabla \rho\|_{L^q(\Omega)}$ and
$\|\nabla F\|_{L^q(\Omega)}$.
Taking  gradient to both sides of  equation $(\ref{equation1})_1$ and $(\ref{equation1})_3$ , we have
\be\label{4.23}
(\nabla\rho)_t+u\cdot\nabla(\nabla\rho)=-\nabla(\rho\nabla\cdot u)-\nabla u\cdot\nabla\rho,
\ee
and
\begin{align}\label{4.24}
\rho(\nabla F)_t+\rho u\cdot\nabla(\nabla F)=\rho\nabla(\nabla uF)-\rho\nabla u\nabla F.
\end{align}
Multiply by $q\rho|\nabla\rho|^{q-2}\nabla\rho$ and $q\rho|\nabla F|^{q-2}\nabla F$  to \eqref{4.23} and \eqref{4.24} respectively, we have
\begin{align}\label{4.25}
\frac{d}{dt}\int_\Omega & \rho|\nabla\rho|^qdx\\\nonumber
=& -q\int_\Omega\nabla(\rho\nabla\cdot u)\rho|\nabla\rho|^{q-2}\nabla\rho dx-q\int_\Omega\nabla u\cdot\nabla\rho \rho|\nabla\rho|^{q-2}\nabla\rho dx\\\nonumber
\leq & C\int_\Omega \rho|\nabla u||\nabla\rho|^q dx+ C\int_\Omega \rho |\nabla^2 u||\nabla\rho|^{q-1}dx\\\nonumber
\leq & C\|\nabla u\|_{L^\infty(\Omega)}\|\rho^{1/q}\nabla\rho\|_{L^q(\Omega)}^q+C\|\nabla^2 u\|_{L^q(\Omega)}\|\rho^{1/q}\nabla\rho\|_{L^q(\Omega)}^{q-1},
\end{align}
as well as
\begin{align}\label{4.26}
\frac{d}{dt}\int_\Omega & \rho|\nabla F|^qdx\\\nonumber
&\leq C\int_{\Omega}|\rho\nabla(\nabla u F)|\nabla F|^{q-2}\nabla F|dx +C\int_{\Omega}|\rho\nabla u\nabla F|\nabla F|^{q-2}\nabla F|dx\\\nonumber
&\leq C\|\nabla^2 u\|_{L^q(\Omega)}\|\rho^{1/q}\nabla F\|_{L^q(\Omega)}^{q-1}+C\|\nabla u\|_{L^{\infty}(\Omega)}\|\rho^{1/q}\nabla F\|^q_{L^q(\Omega)}\\\nonumber
\end{align}
From \eqref{4.25} and \eqref{4.26}, we get
\begin{align}
&\frac{d}{dt}(\|\rho^{1/q}\nabla \rho\|_{L^q(\Omega)}^q+\|\rho^{1/q}\nabla F\|_{L^q(\Omega)}^q)\\\nonumber
\leq & C\|\nabla^2 u\|_{L^q(\Omega)}(\|\rho^{1/q}\nabla \rho\|_{L^q(\Omega)}^{q-1}+\|\rho^{1/q}\nabla F\|_{L^q(\Omega)}^{q-1})\\\nonumber
&+C\|\nabla u\|_{L^\infty(\Omega)}(\|\rho^{1/q}\nabla \rho\|_{L^q(\Omega)}^q+\|\rho^{1/q}\nabla F\|_{L^q(\Omega)}^q).
\end{align}
Noting that
\be
\|\nabla u\|_{L^\infty(\Omega)}\leq \|\nabla v\|_{L^\infty(\Omega)}+ \|\nabla w\|_{L^\infty(\Omega)},
\ee
\be
\|\nabla^2 u\|_{L^q(\Omega)}\leq \|\nabla^2 v\|_{L^q(\Omega)}+ \|\nabla^2 w\|_{L^q(\Omega)},
\ee
and by Lemma 2.2,
\begin{align}
 \|\nabla v\|_{L^\infty(\Omega)}&\leq 1+\|\nabla v\|_{BMO(\Omega)}\ln(e+\|\nabla^2\|_{L^q(\Omega)})\\\nonumber
&\leq 1+ C\ln(e+\|\nabla^2 v\|_{L^q(\Omega)}).
\end{align}
Then from \eqref{4_5},  we obtain
\begin{multline}
\frac{d}{dt}(e+\|\rho^{1/q}\nabla\rho\|_{L^q(\Omega)}+\|\rho^{1/q}\nabla F\|_{L^q(\Omega)})\\
\leq C[1+\ln(e+\|\rho^{1/q}\nabla\rho\|_{L^q(\Omega)}+\|\rho^{1/q}\nabla F\|_{L^q(\Omega)})]\\
(e+\|\rho^{1/q}\nabla\rho\|_{L^q(\Omega)}+\|\rho^{1/q}\nabla F\|_{L^q(\Omega)}).
\end{multline}
Using Gronwall inequality, we get $\|\rho^{1/q}\nabla\rho\|_{L^q(\Omega)}$  and $\|\rho^{1/q}\nabla F\|_{L^q(\Omega)}$  are finite, hence $\|\nabla^2 v\|_{L^q(\Omega)}$ is bounded.

\end{proof}

\begin{remark}\label{4_7}
By (\ref{2.1}), we have $\rho^{-1}=\det F/\rho_0$. When $\rho_0$ has lower bound $\epsilon_0$ and $F\in L^\infty(\Omega)$, we have $\|\rho^{-1}\|_{L^\infty(\Omega)}\leq\frac{\|F\|_{L^\infty(\Omega)}^3}{\epsilon_0}<\infty$.
Namely, if we restrict $\|F\|_{L^\infty(0,T,L^\infty(\Omega))}<\infty$, then $\|\rho^{-1}\|_{L^\infty(0,T,L^\infty(\Omega))}<\infty$
will hold naturally.
\end{remark}
Now we are to prove our main theorem by the contradiction argument.
\\
 {\bf Proof of Theorem 1.2.}
Suppose the maximal existence time of solution $T^\star<\infty$, and
$
\|\rho\|_{L^\infty(0,T^\star, L^\infty(\Omega))}+\|F\|_{L^\infty(0,T^\star, L^{\infty}(\Omega))}< \infty.
$
 By Remark \ref{4_7} and Proposition \ref{prop1}, we obtain
\[
\|\nabla u\|_{L^1(0,T^\star, L^\infty(\Omega))}< \infty.
\]
This is a contradiction with Theorem \ref{theorem1}.

\section*{Acknowledgement}
The authors  express their sincere appreciation to professor Zhen Lei for his
constructive suggestion and discussion. They also thank  Dr. Geng Chen for his kindness help.  The work proceeded substantially while
the first author was visiting the Department of Mathematics, Penn State University. He thanks
deeply the Department of Mathematics for the warm hospitality.
%This work is completed when the first author
%isiting the Department of Mathematics of  Pennsylvania State University, he would express his gratitude to Prof. Chun Liu for his hospitality.
The first author was partly supported by  the NSFC (grant: 11001088), and the Guang-Zhou Pearl River New-Star 2012001.

%%%%%%%%%%%%%%%%%%%%%%%%%%%%%%%%%%%%%%%%%%%%%%%%%%%%%%%%%%%%%%%%%%%%%%%%%%%%%%%%%%%%%%%%%%%%%%%%%%%%%%%%%%%%%%%%%%%%%%%%%%%%%%%%%%%%%%%%%5

\end{document}